\documentclass[reqno,10pt,centertags]{amsart} 
\usepackage{amsmath,amsthm,amscd,amssymb,latexsym,upref,enumerate}
\usepackage{bbm}
\usepackage{nicefrac}
\usepackage{hyperref} 
\newcommand*{\mailto}[1]{\href{mailto:#1}{\nolinkurl{#1}}}
\newcommand{\arxiv}[1]{\href{http://arxiv.org/abs/#1}{arXiv:#1}}

\newcommand{\bbN}{{\mathbb{N}}}
\newcommand{\bbR}{{\mathbb{R}}}

\newcommand{\bbC}{{\mathbb{C}}}

\newcommand{\cB}{{\mathcal B}}

\newcommand{\cH}{{\mathcal H}}

\newcommand{\cN}{{\mathcal N}}

\newcommand{\cW}{{\mathcal W}}
\newcommand{\cX}{{\mathcal X}}

\newcommand{\no}{\notag}
\newcommand{\lb}{\label}

\newcommand{\ol}{\overline}

\newcommand{\wti}{\widetilde}

\newcommand{\Oh}{O}

\newcommand{\f}{\frac}
\newcommand{\bi}{\bibitem}
\newcommand{\hatt}{\widehat}

\renewcommand{\Re}{\mathop\mathrm{Re}}
\renewcommand{\Im}{\mathop\mathrm{Im}}
\renewcommand{\ge}{\geqslant}

\DeclareMathOperator{\dom}{dom}
\DeclareMathOperator{\ran}{ran}

\DeclareMathOperator*{\slim}{s-lim}

\newcommand{\norm}[1]{\left\Vert#1\right\Vert}

\newcommand{\dott}{\,\cdot\,}

\renewcommand{\dotplus}{\overset{\textbf{\Large.}} +}

\allowdisplaybreaks \numberwithin{equation}{section}

\newtheorem{theorem}{Theorem}[section]
\newtheorem{proposition}[theorem]{Proposition}
\newtheorem{lemma}[theorem]{Lemma}
\newtheorem{corollary}[theorem]{Corollary}

\newtheorem{hypothesis}[theorem]{Hypothesis}

\theoremstyle{remark}
\newtheorem{remark}[theorem]{Remark}

\begin{document}

\title{Some Remarks On Krein--von Neumann Extensions}

\author[F.\ Gesztesy]{Fritz Gesztesy}
\address{Department of Mathematics, 
Baylor University, Sid Richardson Bldg., 1410 S.\,4th Street, Waco, TX 76706, USA}
\email{\mailto{Fritz\_Gesztesy@baylor.edu}}
\urladdr{\url{https://math.artsandsciences.baylor.edu/person/fritz-gesztesy-phd}}

\author[S.\ Sukhtaiev]{Selim Sukhtaiev}
\address{Department of Mathematics and Statistics, Auburn University, Auburn, 
AL 36849, USA}
\email{\mailto{szs0266@auburn.edu}}
\urladdr{\url{https://www.auburn.edu/cosam/faculty/math_stats/sukhtaiev/index.htm}}

\thanks{AMP Conference Proceedings, T.\ Aktosun, S.\ Avdonin, R.\ Weder, and V.\ Yurko (eds.), De Gruyter (to appear).}
\date{\today}

\begin{abstract}
We survey various properties of Krein--von Neumann extensions $S_K$ and the reduced Krein--von Neumann operator $\hatt S_K$ in connection with a strictly positive (symmetric) operator $S$ with nonzero deficiency indices. In particular, we focus on the resolvents of $S_K$ and $\hatt S_K$ and of the trace ideal properties of the resolvent of $\hatt S_K$, and make some comparisons with the corresponding properties of the resolvent of the Friedrichs extension $S_F$. We also recall a parametrization of all nonnegative self-adjoint extensions of $S$ and various Krein-type resolvent formulas for any two relatively prime self-adjoint extensions of $S$, utilizing a Donoghue-type $M$-operator (i.e., an energy parameter dependent Dirichlet-to-Neumann-type map). 
\end{abstract}

\maketitle

{\scriptsize \tableofcontents}

\section{Introduction}   \lb{s1}

This paper represents a survey of certain properties of Krein--von Neumann extensions somewhat alternative to those presented in \cite{AGLMS17}, \cite{AGMST10}, \cite{AGMST13}, \cite{FGKLNS22} that focuses on properties of resolvents of the (reduced) Krein--von Neumann and Friedrichs extension of strictly positive operators with nonzero deficiency indices. We emphasize that large parts of this paper are primarily of an expository nature. 

Given a strictly positive (symmetric) operator $S$ in the complex, separable Hilbert space $\cH$, that is, for some $\varepsilon \in (0,\infty)$, $S \geq \varepsilon I_{\cH}$ with deficiency indices $(r,r)$, $r \in \bbN \cup \{\infty\}$, we focus on some properties and interrelations of its two extremal nonnegative self-adjoint extensions, the Friedrichs extension $S_F$, and the Krein--von Neumann extensions $S_K$. In addition, with $P_{\ker(S_K)}$ the orthogonal projection onto the null space of $S_K$, respectively, $S^*$ , since actually, 
\begin{equation} 
P_{\ker(S_K)} = P_{\ker(S^*)},     \lb{1.1} 
\end{equation} 
we study the {\it reduced Krein--von Neumann operator} $\hatt S_K$ in $\hatt \cH$, defined by 
\begin{align}
\hatt S_K:&=S_K|_{[\ker(S_K)]^{\bot}}    \lb{1.2} \\
\begin{split}
& = S_K[I_{\cH} - P_{\ker(S_K)}]  \, \text{ in $[I_{\cH} - P_{\ker(S_K)}]\cH$}    \lb{1.3} \\
&= [I_{\cH} - P_{\ker(S_K)}]S_K[I_{\cH} - P_{\ker(S_K)}] 
\, \text{ in $[I_{\cH} - P_{\ker(S_K)}]\cH$},  
\end{split}
\end{align} 
where 
\begin{equation}
\hatt \cH = [\ker (S^*)]^{\bot} = [I_{\cH} - P_{\ker(S^*)}] \cH 
= [I_{\cH} - P_{\ker(S_K)}] \cH = [\ker (S_K)]^{\bot}.    \lb{1.4}
\end{equation}

In Section \ref{s2} we first discuss the basic facts underlying $S_F$, $S_K$, and $\hatt S_K$, and then also recall a parametrization of all nonnegative self-adjoint extensions $S_{B,\cW}$ of $S$ in terms of nonnegative, self-adjoint operators $B \geq 0$ in $\cW$, $B: \dom(B) \to \cW$, $\dom(B) \subseteq \cW$, where $\cW$ is a closed linear subspace of $\ker(S^*) = \ker(S_K)$. We conclude Section \ref{s2} by recalling the basic inequalities between the eigenvalues of $S_F$ and $\hatt S_K$. Various Krein-type resolvent formulas for a pair of relatively prime self-adjoint extensions of $S$ expressed in terms of an appropriate Donoghue-type $M$-operator (i.e., an energy parameter dependent Dirichlet-to-Neumann-type map and an operator-valued Nevanlinna--Herglotz function) are discussed in Section \ref{s3}.
Finally, in Section \ref{s4} we draw several comparisons between the abstract Friedrichs extension and the (reduced) Krein--von Neumann extension regarding trace ideal properties of their resolvents. For instance, Theorem \ref{t4.1} recalls the known result  
\begin{equation}
(S_F)^{-1} \in \cB_p(\cH) \, \text{ implies } \, \big(\hatt S_K\big)^{-1} \in \cB_p\big(\hatt \cH\big); \quad p \in (0,\infty],  \lb{1.5} 
\end{equation}
and Theorem \ref{t4.3} shows that for $p \in (0,\infty]$, the following items $(i)$--$(iv)$ are equivalent: \\[1mm]
$(i)$ \hspace*{0.5mm} $\big(\hatt S_K\big)^{-1} \in \cB_p\big(\hatt \cH\big)$. \\[1mm] 
$(ii)$ \hspace*{0.001mm} $[I_{\cH} - P_{\ker(S_K)}] (S_F)^{-1} [I_{\cH} - P_{\ker(S_K)}] \in \cB_p(\cH)$. \\[1mm] 
$(iii)$ $[I_{\cH} - P_{\ker(S_K)}] (S_F)^{-1/2} \in \cB_{2p}(\cH)$. \\[1mm] 
$(iv)$ \hspace*{0.00001mm}  $(S_F)^{-1/2} [I_{\cH} - P_{\ker(S_K)}] \in \cB_{2p}(\cH)$. 

Employing the decomposition of $\cH$ into 
\begin{equation}
\cH = \ker(S_K) \oplus \hatt \cH = \ker(S^*) \oplus \hatt \cH 
\end{equation}
(cf.\ \eqref{1.4}), $(S_F)^{-1}$ decomposes accordingly into the $2 \times 2$ block operator matrix 
\begin{align}
& (S_F)^{-1}      \lb{1.6} \\
& \quad = \begin{pmatrix} P_{\ker(S_K)} (S_F)^{-1} P_{\ker(S_K)} & 
P_{\ker(S_K)} (S_F)^{-1} [I_{\cH} - P_{\ker(S_K)}] \\ 
[I_{\cH} - P_{\ker(S_K)}] (S_F)^{-1} P_{\ker(S_K)} & 
[I_{\cH} - P_{\ker(S_K)}] (S_F)^{-1} [I_{\cH} - P_{\ker(S_K)}] 
\end{pmatrix}.    \no
\end{align}
Thus, assuming Hypothesis \ref{h2.7} and $\big(\hatt S_K\big)^{-1} \in \cB_p\big(\hatt \cH\big)$, three of the four entries in the $2 \times 2$ block operator matrix \eqref{1.6} lie in $\cB_p(\cH)$, the lone exception being the $(1,1)$-entry, 
$P_{\ker(S_K)} (S_F)^{-1} P_{\ker(S_K)}$ in \eqref{1.6}. In particular, the converse to assertion \eqref{1.5} does not hold, in other words, the hypothesis 
\begin{align} 
& \big(\hatt S_K\big)^{-1} \in \cB_p\big(\hatt \cH\big),    \no \\ 
& \quad \text{equivalently, } \\ 
& [I_{\cH} - P_{\ker(S_K)}] (S_K)^{-1} [I_{\cH} - P_{\ker(S_K)}] \in \cB_p(\cH); \quad p \in (0,\infty],    \no 
\end{align} 
cannot make any assertions about trace ideal properties of the remaining $(1,1)$-entry $P_{\ker(S_F)} (S_K)^{-1} P_{\ker(S_K)}$ in \eqref{1.6}, a fact discussed in detail by somewhat different means in \cite{HMD04}, \cite{Ma92a}, \cite{Ma23}, \cite{Ma23a}, and \cite{Ma24}.  

Finally, we briefly elaborate on the notational conventions used
throughout this paper: Let $\cH$ be a separable complex Hilbert space, $(\cdot,\cdot)_{\cH}$ the scalar product in $\cH$ (linear in
the second factor), and $I_{\cH}$ the identity operator in $\cH$.
Next, let $T$ be a linear operator mapping (a linear subspace of) a complex, separable 
Hilbert space into another, with $\dom(T)$ and $\ran(T)$ denoting the
domain and range of $T$. The closure of a closable operator $S$ is
denoted by $\ol S$. The kernel (i.e., null space) of $T$ is denoted by
$\ker(T)$. The spectrum and resolvent set of a closed linear operator in $\cH$ will be denoted by $\sigma(\cdot)$ and $\rho(\cdot)$, respectively (the essential spectrum for self-adjoint operators is abbreviated by $\sigma_{ess}(\dott)$). The
Banach spaces of bounded and compact linear operators in $\cH$ are
denoted by $\cB(\cH)$ and $\cB_\infty(\cH)$ (and analogously by $\cB(\cH_1,\cH_2)$ and $\cB_\infty(\cH_1,\cH_2)$ in a two-Hilbert space context). Similarly,
the $\ell^p$-based Schatten--von Neumann (trace) ideals (cf.\ \cite[Chs.~1, 2]{Si05} will be denoted
by $\cB_p(\cH)$, $p\in (0,\infty)$. 

The symbol $\dotplus$ abbreviates the direct (not necessarily orthogonal direct) sum of linear subspaces in $\cH$. Finally, $\cX_1 \hookrightarrow \cX_2$ denotes the continuous embedding of the normed space $\cX_1$ into the normed space $\cX_2$.

\section{On the Abstract Friedrichs and Krein--von Neumann Extension}   \lb{s2}

In this section we recall various facts on the abstract Krein--von Neumann extension. 

To set the stage, a linear operator $S:\dom(S)\subseteq\cH\to\cH$, is called {\it symmetric}, if
\begin{equation}\lb{Pos-2}
(u,Sv)_\cH=(Su,v)_\cH, \quad u,v\in \dom (S).
\end{equation}
If $\dom(S)=\cH$, the classical Hellinger--Toeplitz theorem guarantees that $S\in\cB(\cH)$, in which situation $S$ is readily seen to be self-adjoint. In general, however, symmetry is a considerably weaker property than self-adjointness and a classical problem in functional analysis is that of determining all self-adjoint extensions of a given unbounded symmetric operator of equal and nonzero deficiency indices. In this manuscript we will be particularly interested in this question within the class of 
densely defined \big(i.e., $\ol{\dom(S)}=\cH$\big), non-negative operators (in fact, in most instances $S$ will even turn out to be strictly positive) and we focus almost exclusively on self-adjoint extensions that are non-negative operators.  

We recall that a linear operator $S:\dom(S)\subseteq\cH\to \cH$
is called {\it non-negative} provided
\begin{equation}\lb{Pos-1}
(u,Su)_\cH\geq 0, \quad u\in \dom(S).
\end{equation}
(In particular, $S$ is symmetric in this case.) $S$ is called {\it strictly positive}, if for some 
$\varepsilon >0$, $(u,Su)_\cH\geq \varepsilon \|u\|_{\cH}^2$, $u\in \dom(S)$. 
Next, we recall that $A \leq B$ for two self-adjoint operators in $\cH$ if 
\begin{align}
\begin{split}
& \dom(|A|^{1/2}) \supseteq \dom(|B|^{1/2}) \, \text{ and } \\ 
& (|A|^{1/2}u, U_A |A|^{1/2}u)_{\cH} \leq (|B|^{1/2}u, U_B |B|^{1/2}u)_{\cH}, \;  
u \in \dom(|B|^{1/2}),      \lb{AleqB} 
\end{split}
\end{align}
where $U_C$ denotes the partial isometry in $\cH$ in the polar decomposition of 
a densely defined closed operator $C$ in $\cH$, $C=U_C |C|$, $|C|=(C^* C)^{1/2}$. (If,  
in addition, $C$ is self-adjoint, then $U_C$ and $|C|$ commute.) 
We also recall (see, e.g., \cite[Section\ I.6]{Fa75}, \cite[Theorem\ VI.2.21]{Ka80}) that if $A$ and $B$ are both 
self-adjoint and nonnegative in $\cH$, then 
\begin{equation}
0 \leq A\leq B  \, \text{ if and only if } \, (B + a I_\cH)^{-1} \leq (A + a I_\cH)^{-1} 
\, \text{ for all $a>0$,}     \lb{PPa-1} 
\end{equation}
(which implies $0 \leq A^{1/2 }\leq B^{1/2}$) and
\begin{equation}
\ker(A) =\ker\big(A^{1/2}\big)
\end{equation}
(with $C^{1/2}$ the unique nonnegative square root of a nonnegative self-adjoint operator $C$ in $\cH$).

For simplicity we will always adhere in the following to the conventions that $S$ is a linear, unbounded, densely defined, nonnegative (i.e., $S\geq 0$) operator in $\cH$, and that $S$ has nonzero deficiency indices, that is,
\begin{equation}
{\rm def} (S) = \dim (\ker(S^*-z I_{\cH})) \in \bbN\cup\{\infty\}, 
\quad z\in \bbC\backslash [0,\infty). 
\lb{DEF}
\end{equation}
Moreover, since $S$ and its closure $\ol{S}$ have the same self-adjoint extensions in $\cH$, we will without loss of generality assume that $S$ is closed in the remainder of this paper.

\begin{hypothesis} \lb{h2.1}
Assume that $S$ is a densely defined, closed, nonnegative operator in $\cH$ with deficiency indices $(r,r)$, $r \in \bbN \cup \{\infty\}$.
\end{hypothesis}

The following is a fundamental result to be found in M.\ Krein's celebrated 1947 paper
 \cite{Kr47} (see also the English summary on pages 491--495): 

\begin{theorem}\lb{t2.2}
Assume Hypothesis \ref{h2.1}. Then, among all non-negative self-adjoint 
extensions of $S$, there exist two distinguished ones, $S_K$ and $S_F$, which are, respectively, the smallest and largest
$($in the sense of order between linear operators, cf.\ \eqref{PPa-1}$)$ such extensions. Furthermore, any non-negative self-adjoint extension $\widetilde{S}$ of $S$ necessarily satisfies
\begin{equation}\lb{Fr-Sa}
S_K\leq\widetilde{S}\leq S_F.
\end{equation}
In particular, \eqref{Fr-Sa} determines $S_K$ and $S_F$ uniquely. \\
In addition,  if $S\geq \varepsilon I_{\cH}$ for some $\varepsilon >0$, one has 
$S_F \geq \varepsilon I_{\cH}$, and  
\begin{align}
\dom (S_F) &= \dom (S) \dotplus (S_F)^{-1} \ker (S^*),     \lb{SF}  \\
\dom (S_K) & = \dom (S) \dotplus \ker (S^*),    \lb{SK}   \\
\dom (S^*) & = \dom (S) \dotplus (S_F)^{-1} \ker (S^*) \dotplus \ker (S^*)  \no \\
& = \dom (S_F) \dotplus \ker (S^*),    \lb{S*} 
\end{align}
in particular, 
\begin{equation} \lb{Fr-4Tf}
\ker(S_K)= \ker\big((S_K)^{1/2}\big)= \ker(S^*) = \ran(S)^{\bot}.
\end{equation} 
\end{theorem}

Here the operator inequalities in \eqref{Fr-Sa} are understood in the sense of \eqref{AleqB} and hence they can  equivalently be written as
\begin{equation}
(S_F + a I_{\cH})^{-1} \leq \big(\wti S + a I_{\cH}\big)^{-1} \leq (S_K + a I_{\cH})^{-1} 
\, \text{ for all $a > 0$.}    \lb{Res}
\end{equation}

We will call the operator $S_K$ the {\it Krein--von Neumann extension}
of $S$. See \cite{Kr47} and also the discussion in \cite{AS80}. It should be
noted that the Krein--von Neumann extension was first considered by von Neumann 
\cite{Ne29} in 1929 in the case where $S$ is strictly bounded from below, that is, if 
$S \geq \varepsilon I_{\cH}$ for some $\varepsilon >0$. (His construction appears in the proof of Theorem 42 on pages 102--103.) However, von Neumann did not isolate the extremal property of this extension as described in \eqref{Fr-Sa} and \eqref{Res}. M.\ Krein \cite{Kr47}, \cite{Kr47a} was the first to systematically treat the general case $S\geq 0$ and to study all nonnegative self-adjoint extensions of $S$, illustrating the special role of the {\it Friedrichs extension} (i.e., the ``hard'' extension) $S_F$ of $S$ and the Krein--von Neumann extension (i.e., the ``soft'' extension)  $S_K$ of $S$ as extremal cases when considering all nonnegative extensions of $S$. 

For classical references on the subject of self-adjoint extensions of semibounded operators (not necessarily restricted to the Krein--von Neumann extension) we refer to Birman \cite{Bi56}, \cite{Bi08}, Friedrichs \cite{Fr34}, Freudenthal \cite{Fr36}, Krein \cite{Kr47a}, {\u S}traus \cite{St73}, and Vi{\u s}ik \cite{Vi63} (see also the monographs \cite[Sect. 109]{AG81a} and \cite[Part III]{Fa75}). For more recent references in this context, see the end of this section.  

An intrinsic description of the Friedrichs extension $S_F$ of $S\geq 0$, due to Freudenthal \cite{Fr36} in 1936, is that the unbounded operator $S_F:\dom(S_F)\subset\cH\to\cH$ is given by   
\begin{align}
& S_F u =S^*u,   \no \\
& u \in \dom(S_F) = \big\{v\in\dom(S^*)\,\big|\,  \mbox{there exists} \, 
\{v_j\}_{j\in\bbN}\subset \dom(S),    \lb{Fr-2} \\
& \quad \mbox{with} \, \lim_{j\to\infty}\|v_j-v\|_{\cH}=0  
\mbox{ and } ((v_j-v_k),S(v_j-v_k))_\cH\to 0 \mbox{ as }  j,k\to\infty\big\}.   \no 
\end{align}
Then, as is well-known, $S_F$ is non-negative, 
\begin{align}
& S_F \geq 0, \quad S\subseteq S_F=S_F^*\subseteq S^*,  \lb{Fr-4}  \\
& \dom\big((S_F)^{1/2}\big)=\big\{v\in\cH\,\big|\, \mbox{there exists} \, 
\{v_j\}_{j\in\bbN}\subset \dom(S),   \lb{Fr-4J} \\
& \quad \mbox{with} \lim_{j\to\infty}\|v_j-v\|_{\cH}=0  
\mbox{ and } ((v_j-v_k),S(v_j-v_k))_\cH\to 0\mbox{ as }
j,k\to\infty\big\},  \no 
\end{align}
and
\begin{equation}\lb{Fr-4H}
S_F=S^*|_{\dom(S^*)\cap\dom((S_{F})^{1/2})}.
\end{equation}

Equations \eqref{Fr-4J} and \eqref{Fr-4H} are intimately related to the definition of $S_F$ via (the closure of) the sesquilinear form generated by $S$ as follows: One introduces the sesquilinear form
\begin{equation}
q_S(f,g)=(f,Sg)_{\cH}, \quad f, g \in \dom(q_S)=\dom(S). 
\end{equation}
Since $S\geq 0$, the form $q_S$ is closable and we denote by $Q_S$ the closure of 
$q_S$. Then $Q_S\geq 0$ is densely defined and closed. By the first and second representation theorem for forms (cf., e.g., \cite[Sect.\ 6.2]{Ka80}), $Q_S$ is uniquely associated with a nonnegative, self-adjoint operator in $\cH$. This operator is precisely the Friedrichs extension, $S_F \geq 0$, of $S$, and hence,
\begin{align}
& Q_S(f,g)=(f,S_F g)_{\cH}, \quad f \in \dom(Q_S), \, g \in \dom(S_F),  \lb{Fr-Q} \\ 
& \dom(Q_S) = \dom\big(S_F^{1/2}\big).    \lb{Fr-QS} 
\end{align}

An intrinsic description of the Krein--von Neumann extension $S_K$ of $S\geq 0$ has been given by Ando and Nishio \cite{AN70} in 1970, where $S_K$ has been characterized as the unbounded operator $S_K:\dom(S_K)\subset\cH\to\cH$ given by
\begin{align} 
& S_Ku = S^*u,   \no \\
& u \in \dom(S_K) = \big\{v\in\dom(S^*)\,\big|\,\mbox{there exists} \, 
\{v_j\}_{j\in\bbN}\subset \dom(S),    \lb{Fr-2X}  \\ 
& \quad \mbox{with} \, \lim_{j\to\infty} \|Sv_j-S^*v\|_{\cH}=0  
\mbox{ and } ((v_j-v_k),S(v_j-v_k)_\cH\to 0 \mbox{ as } j,k\to\infty\big\}.  \no
\end{align}

One observes that by \eqref{Fr-2} one has that shifting $S$ by a constant commutes with the operation of taking the Friedrichs extension of $S$, that is, for any 
$c\in\bbR$,
\begin{equation}
(S + c I_{\cH})_{F} = S_F + c I_{\cH},    \lb{Fr-c}
\end{equation}
but by \eqref{Fr-2X}, the analog of \eqref{Fr-c} for the Krein--von Neumann extension 
$S_K$ fails.

At this point we recall a result due to Makarov and Tsekanovskii \cite{MT07}, concerning symmetries (e.g., a rotational symmetry) shared by $S$, $S^*$, $S_F$, and $S_K$ (see also \cite{HK09}). More precisely, we assume that the symmetry is represented by a unitary operator $U$ in $\cH$ commuting with $S$, that is,
\begin{equation} 
 U S U^{-1} = S, \, 
 \text{ or equivalently, } \, U \dom(S) \subseteq \dom(S) \, \text{ and } \, U S \subseteq SU.   
\end{equation}
(Incidentally, we note that this implies $U \dom(S) = \dom(S)$.)

\begin{proposition}  \lb{p2.3}
Let $U$ be unitary in $\cH$ and assume $S$ to be a densely defined, closed, nonnegative operator in $\cH$ that commutes with $U$, that is, 
\begin{align}
U S U^{-1} &= S.  \lb{US}
\intertext{Then also $S^*$, $S_F$, and $S_K$ commute with $U$, that is,}
U S^* U^{-1} &= S^*,    \lb{US*} \\
U S_F U^{-1} &= S_F,  \lb{USF}\\
U S_K U^{-1} &= S_K.  \lb{USK} 
\end{align}
\end{proposition} 

In fact, Makarov and Tsekanovskii \cite{MT07} prove an extension of 
Proposition \ref{p2.3} where \eqref{US} is replaced by $U S U^{-1} = \mu S$, for some 
$\mu > 0$, and analogously in \eqref{US*}--\eqref{USK}. Here we just note that 
\eqref{US*} is a consequence of \cite[p.\ 73, 74]{We80}, and \eqref{USF}, respectively,  
\eqref{USK}, immediately follow from \eqref{US*} and \eqref{Fr-2}, respectively, 
 \eqref{Fr-2X}.
 
Similarly to Proposition \ref{p2.3}, the following results also immediately follows from the characterizations \eqref{Fr-2} and \eqref{Fr-2X} of $S_F$ and $S_K$, respectively:

\begin{proposition}  \lb{p2.4}  
Let $U\colon\cH_1\to\cH_2$ be unitary from $\cH_1$ onto $\cH_2$ and assume $S$ to be a densely defined, closed, nonnegative operator in $\cH_1$ with adjoint 
$S^*$, Friedrichs extension $S_F$, and Krein--von Neumann extension $S_K$ in 
$\cH_1$, respectively. Then the adjoint, Friedrichs extension, and Krein--von Neumann extension of the nonnegative, closed, densely defined, symmetric operator $USU^{-1}$ in $\cH_2$ are given by 
\begin{align}
[USU^{-1}]^* = US^*U^{-1},  \quad 
[USU^{-1}]_F = US_F U^{-1},  \quad 
[USU^{-1}]_K = US_K U^{-1}
\end{align}
in $\cH_2$, respectively.
\end{proposition} 

\begin{proposition} [see, e.g., \cite{MN12}, Corollary~3.10]  \lb{p2.5}  
Let $J\subseteq \bbN$ be some countable index set and consider $\cH = \bigoplus_{j\in J} \cH_j$ and $S=\bigoplus_{j\in J} S_j$, where each $S_j$ is a densely defined, closed, nonnegative operator in $\cH_j$, $j\in J$. Denoting by $(S_j)_F$ and $(S_j)_K$ the Friedrichs and Krein--von Neumann extension of $S_j$ in $\cH_j$, $j\in J$, one infers that 
\begin{equation}
S^*=\bigoplus_{j\in J} \; (S_j)^*, \quad S_F=\bigoplus_{j\in J} \; (S_j)_F, 
\quad S_K = \bigoplus_{j\in J} \; (S_j)_K. 
\end{equation}
\end{proposition} 

The following is a consequence of a slightly more general result formulated
in \cite[Theorem 1]{AN70}: 

\begin{proposition}\lb{p2.6}
Assume Hypothesis \ref{h2.1}. 
Then $S_K$, the Krein--von Neumann extension of $S$, has the
property that
\begin{equation}\lb{an-T1}
\dom\big((S_K)^{1/2}\big)=\biggl\{u\in\cH\,\bigg|\,\sup_{v\in\dom(S)}
\frac{|(u,Sv)_\cH|^2}{(v,Sv)_\cH}
<+\infty\biggr\},
\end{equation}
and
\begin{equation}\lb{an-T2}
\big\|(S_K)^{1/2}u\big\|^2_\cH=\sup_{v\in\dom(S)}
\frac{|(u,Sv)_\cH|^2}{(v,Sv)_\cH}, \quad  u\in\dom\big((S_K)^{1/2}\big).  
\end{equation}
\end{proposition}

A word of explanation is in order here: Given $S\geq 0$ as in the statement of
Proposition \ref{p2.6}, the Cauchy--Schwarz-type inequality
\begin{equation}\lb{CS-I.1}
|(u,Sv)_{\cH}|^2\leq (u,Su)_{\cH} (v,Sv)_{\cH}, \quad u,v\in\dom(S), 
\end{equation}
shows (due to the fact that $\dom(S)\hookrightarrow\cH$ densely, equipping $\dom(S)$ with the graph norm induced by $S$) that
\begin{equation}\lb{CS-I.2}
u\in\dom(S) \,\mbox{ and } \, (u,S u)_{\cH} =0 \,
\text{ implies }\,Su=0.
\end{equation}
Thus, whenever the denominator of the fractions appearing
in \eqref{an-T1}, \eqref{an-T2} vanishes, so does the numerator, and
one interpretes $0/0$ as being zero in \eqref{an-T1}, \eqref{an-T2}.

We continue by recording an abstract result regarding the
parametrization of all non-negative self-adjoint extensions of a given
strictly positive, densely defined, symmetric operator. 

\begin{hypothesis} \lb{h2.7}
Suppose that $S$ is a densely defined, closed operator in $\cH$, for some $\varepsilon>0$, 
$S \geq \varepsilon I_{\cH}$, and $S$ has deficiency indices $(r,r)$, $r \in \bbN \cup \{\infty\}$. 
\end{hypothesis}

The following results are proved in \cite{AS80}, \cite[\S 15]{Fa75}, 
\cite{Gr68}, \cite{Gr70}, \cite{Ma10}, and, in the present form, the next theorem 
appeared in \cite{AS80}, \cite{Bi56}, \cite{GM11}, \cite{Ma92a}, \cite[Ch.~13]{Sc12}:

\begin{theorem}\lb{t2.8}
Assume Hypothesis \ref{h2.7}. Then there exists a one-to-one correspondence between non-negative self-adjoint operators 
$0 \leq B:\dom(B)\subseteq \cW\to \cW$, $\ol{\dom(B)}=\cW$, where $\cW$
is a closed, linear subspace of $\cN_0 = \ker(S^*)$, and non-negative self-adjoint
extensions $S_{B,\cW}\geq 0$ of $S$. More specifically, $S_F$ is boundedly invertible, 
$S_F\geq \varepsilon I_{\cH}$, 
and one has
\begin{align} 
& \dom(S_{B,\cW})  \no \\
& \quad =\big\{f + (S_F)^{-1}(Bw + \eta)+ w \,\big|\,
f\in\dom(S),\, w\in\dom(B),\, \eta\in \cN_0 \cap \cW^{\bot}\big\},  \no \\
& S_{B,\cW} = S^*|_{\dom(S_{B,\cW})},       \lb{AS-2}  
\end{align} 
where $\cW^{\bot}$ denotes the orthogonal complement of $\cW$ in $\cN_0$. In 
addition,
\begin{align}
&\dom\big((S_{B,\cW})^{1/2}\big) = \dom\big((S_F)^{1/2}\big) \dotplus 
\dom\big(B^{1/2}\big),   \\
&\big\|(S_{B,\cW})^{1/2}(g+u)\big\|_{\cH}^2 =\big\|(S_F)^{1/2} g\big\|_{\cH}^2 
+ \big\|B^{1/2} u\big\|_{\cH}^2, \\ 
& \hspace*{2.3cm} g \in \dom\big((S_F)^{1/2}\big), \; u \in \dom\big(B^{1/2}\big),  \no
\end{align}
implying,
\begin{equation}\lb{K-ee}
\ker(S_{B,\cW})=\ker(B). 
\end{equation} 
Moreover\footnote{In the case where $\wti \cW \subsetneqq \cW$, we view $\wti B$ as a nondensely defined operator in $\cW$ when alluding to $B \leq \wti B$ in the sense that $\big(\wti f,B \wti f\big)_{\cW} \leq \big(\wti f, {\wti B} \wti f\big)_{\cW}$, $\wti f \in \dom\big(\wti B\big) \subset \dom(B)$.}, 
\begin{equation}
B \leq \wti B \, \text{ implies } \, S_{B,\cW} \leq S_{\wti B,\wti \cW},
\end{equation}
where 
\begin{align}
\begin{split}
& B\colon \dom(B) \subseteq \cW \to \cW, \quad 
 \wti B\colon \dom\big(\wti B\big) \subseteq \wti \cW \to \wti \cW,   \\
& \ol{\dom\big(\wti B\big)} = \wti\cW \subseteq \cW = \ol{\dom(B)}. 
\end{split}
\end{align}

In the above scheme, the Krein--von Neumann extension $S_K$
of $S$ corresponds to the choice $\cW=\cN_0$ and $B=0$ $($with 
$\dom(B)=\dom(B^{1/2})=\cN_0=\ker (S^*)$$)$.
In particular, one thus recovers \eqref{SK}, and \eqref{Fr-4Tf}, and also obtains
\begin{align}
&\dom\big((S_{K})^{1/2}\big) = \dom\big((S_F)^{1/2}\big) \dotplus \ker (S^*),  
\lb{SKform1}  \\
&\big\|(S_{K})^{1/2}(g+u)\big\|_{\cH}^2 =\big\|(S_F)^{1/2} g\big\|_{\cH}^2, 
\quad g \in \dom\big((S_F)^{1/2}\big), \; u \in \ker (S^*).   \lb{SKform2}
\end{align}
Finally, the Friedrichs extension $S_F$ corresponds to the choice $\dom(B) = \cW = \{0\}$ $($i.e., formally, 
$B\equiv\infty$$)$, in which case one recovers \eqref{SF}. 
\end{theorem}

For subsequent purposes we also note that under the assumptions on $S$ in Theorem \ref{t2.8}, one has 
\begin{equation}
\dim(\ker (S^*-z I_{\cH})) = \dim(\ker(S^*)) = \dim (\cN_0) = {\rm def} (S), 
\quad z\in \bbC\backslash [\varepsilon,\infty).   \lb{dim}
\end{equation}

We recall that two self-adjoint extensions $S_1$ and $S_2$ of $S$ are called {\it relatively prime} if $\dom (S_1) \cap \dom (S_2) = \dom (S)$. As an immediate consequence of \eqref{SF}, \eqref{SK}, $S_F \geq \varepsilon I_{\cH}$, and hence $(S_F)^{-1} \ker(S^*) \cap \ker(S^*) = \{0\}$, one obtains the following fact: 

\begin{lemma}  \lb{l2.9}
Assume Hypothesis \ref{h2.7}. Then $S_F$ and $S_K$ are 
relatively prime, that is,
\begin{equation}
\dom (S_F) \cap \dom (S_K) = \dom (S).    \lb{RP}
\end{equation}
\end{lemma}

Next, we consider a self-adjoint operator
\begin{equation} \lb{Barr-1}
T:\dom(T)\subseteq \cH\to\cH,\quad T=T^*,
\end{equation} 
which is bounded from below, that is, there exists $\alpha\in\bbR$ such that
\begin{equation} \lb{Barr-2}
T\geq \alpha I_{\cH}.
\end{equation} 
We denote by $\{E_T(\lambda)\}_{\lambda\in\bbR}$ the family of strongly right-continuous spectral projections of $T$, and introduce $E_T((a,b))=E_T(b_-) - E_T(a)$, 
$E_T(b_-) = \slim_{\varepsilon\downarrow 0}E_T(b-\varepsilon)$, $-\infty \leq a < b$. In addition, we set 
\begin{equation} \lb{Barr-3}
\mu_{T,j} = \inf\,\bigl\{\lambda\in\bbR\,|\,
\dim (\ran (E_T((-\infty,\lambda)))) \geq j\bigr\},\quad j\in\bbN.
\end{equation} 
Then, for fixed $k\in\bbN$, either: \\
$(i)$ $\mu_{T,k}$ is the $k$th eigenvalue of $T$ counting multiplicity below the bottom of the essential spectrum, $\sigma_{ess}(T)$, of $T$, \\
or \\
$(ii)$ $\mu_{T,k}$ is the bottom of the essential spectrum of $T$, 
\begin{equation}
\mu_{T,k} = \inf \{\lambda \in \bbR \,|\, \lambda \in \sigma_{ess}(T)\}, 
\end{equation}
and in that case $\mu_{T,k+\ell} = \mu_{T,k}$, $\ell\in\bbN$, and there are at most $k-1$ eigenvalues (counting multiplicity) of $T$ below $\mu_{T,k}$. 

We now record a basic result of M.\ Krein \cite{Kr47} with an important extension due 
to Alonso and Simon \cite{AS80}. For this purpose we introduce the {\it reduced  
Krein--von Neumann operator} $\hatt S_K$ in the Hilbert space $\hatt \cH$ (cf.\ \eqref{Fr-4Tf})
\begin{equation}
\hatt \cH = [\ker (S^*)]^{\bot} = [I_{\cH} - P_{\ker(S^*)}] \cH 
= [I_{\cH} - P_{\ker(S_K)}] \cH = [\ker (S_K)]^{\bot},    \lb{hattH}
\end{equation}
by 
\begin{align}
\hatt S_K:&=S_K|_{[\ker(S_K)]^{\bot}}    \lb{Barr-4} \\
\begin{split}
& = S_K[I_{\cH} - P_{\ker(S_K)}]  \, \text{ in $[I_{\cH} - P_{\ker(S_K)}]\cH$}    \lb{SKP} \\
&= [I_{\cH} - P_{\ker(S_K)}]S_K[I_{\cH} - P_{\ker(S_K)}] 
\, \text{ in $[I_{\cH} - P_{\ker(S_K)}]\cH = \hatt \cH$},  
\end{split}
\end{align} 
where $P_{\ker(S_K)}$ denotes the orthogonal projection onto $\ker(S_K) = \ker(S^*)$, and we are alluding to the orthogonal direct sum decomposition of $\cH$ into 
\begin{equation}
\cH = P_{\ker(S_K)}\cH \oplus [I_{\cH} - P_{\ker(S_K)}]\cH = P_{\ker(S_K)}\cH \oplus \hatt \cH.   \lb{2.50}
\end{equation}

The following result is due to M.\ Krein \cite{Kr47}; we provide its elementary proof:

\begin{theorem} [\cite{Kr47}, Theorems~12 and 13] \lb{t3.13}
Assume Hypothesis \ref{h2.7} and let $\wti S$ be any self-adjoint extension $S$ such that 
$0 \in \rho \big(\wti S\big)$. Then  
\begin{equation}
\lb{3.28a}
\big(\hatt S_K\big)^{-1} = 
\left[I_{\cH}-P_{\ker(S_K)}\right]\big(\wti S\big)^{-1}\left[I_{\cH}-P_{\ker(S_K)}\right] 
\, \text{ in $\hatt \cH = [I_{\cH} - P_{\ker(S_K)}]\cH$}. 
\end{equation}
In particular, \eqref{3.28a} holds for $\wti S = S_F$, that is,
\begin{equation}
\lb{2.FK}
\big(\hatt S_K\big)^{-1} = 
\left[I_{\cH}-P_{\ker(S_K)}\right](S_F)^{-1}\left[I_{\cH}-P_{\ker(S_K)}\right] 
\, \text{ in $\hatt \cH = [I_{\cH} - P_{\ker(S_K)}]\cH$}. 
\end{equation}
\end{theorem}
\begin{proof}
It suffices to show that for any $g\in \ran (I_{\cH}-P_{\ker(S^*)})$, 
\begin{equation}
\lb{3.28b}
\big(\hatt S_K\big)^{-1} g= [I_{\cH}-P_{\ker(S^*)}] \big(\wti S\big)^{-1} g 
= [I_{\cH}-P_{\ker(S^*)}]\big(\wti S\big)^{-1} [I_{\cH}-P_{\ker(S^*)}] g.
\end{equation}
Since by hypothesis $S\geq \varepsilon I_{\cH}$, one obtains that $\ran(S)=\ol{\ran(S)}$ is closed 
(see, e.g., \cite[Theorem~1.1' and the Corollary to Theorem~1.7]{Br59}, \cite[Proposition~2.1\,$(iv)$]{Sc12}) and hence $\ran(I_{\cH}-P_{\ker(S^*)})=\ran(S)$ yields $g\in \ran(S)$. Introducing 
$f_1\in \dom(S) \subset \dom(S^*)$ such that $g = S f_1$, and defining 
$f= \big(\hatt S_K\big)^{-1} g$, one concludes that $f\in \dom(S_K) \subset \dom(S^*)$, and hence, 
$(f_1 - f) \in \dom(S^*)$. Since $S_K$ and $\wti S$ are both extensions of $S$ one 
obtains 
\begin{align}
& S^*f_1=Sf_1= \wti S f_1=S_Kf_1=g,     \lb{3.28c} \\
& S^*f = S_Kf = \hatt S_Kf = g.     \lb{3.28d} 
\end{align}
Equations \eqref{3.28c} and \eqref{3.28d} imply 
\begin{equation}
\lb{3.28e3}
S^*(f_1-f)=0, \, \text{ or equivalently, } \, (f_1 - f) \in \ker(S^*). 
\end{equation}
Finally, since $f \in \ran (I_{\cH}-P_{\ker(S^*)})$, one infers that 
\begin{align}
 f& = [I_{\cH}-P_{\ker(S^*)}] f  
 = [I_{\cH}-P_{\ker(S^*)}] (f_1 + (f - f_1)) 
 = [I_{\cH}-P_{\ker(S^*)}]f_1   \no \\ 
 & = [I_{\cH}-P_{\ker(S^*)}]\big(\wti S\big)^{-1}g.    \lb{3.28e1} 
\end{align}
Since $ f = \big(\hatt S_K\big)^{-1} g$, one obtains \eqref{3.28b}, and hence \eqref{3.28a}.
\end{proof}

\begin{theorem} [\cite{AS80}, \cite{Kr47}] \lb{t2.11}
Assume Hypothesis \ref{h2.7}. Then, 
\begin{equation}\lb{Barr-5}
\varepsilon \leq \mu_{S_F,j} \leq \mu_{\hatt S_K,j}, \quad j\in\bbN.
\end{equation} 
In particular, if the Friedrichs extension $S_F$ of $S$ has purely discrete
spectrum, then, except possibly for $\lambda=0$, the Krein--von Neumann extension
$S_K$ of $S$ also has purely discrete spectrum in $(0,\infty)$, that is, 
\begin{equation}
\sigma_{ess}(S_F) = \emptyset \, \text{ implies } \, 
\sigma_{ess}(S_K) \backslash\{0\} = \emptyset.      \lb{ESSK}
\end{equation}
Equivalently, \eqref{ESSK} can be rephrased as 
\begin{align}
\begin{split}
& (S_F - z_0 I_{\cH})^{-1} \in \cB_{\infty}(\cH) 
\, \text{ for some $z_0\in \bbC\backslash [\varepsilon,\infty)$ implies}   \\ 
& \quad 
(S_K - zI_{\cH})^{-1}[I_{\cH} - P_{\ker(S_K)}] \in \cB_{\infty}(\cH) 
 \, \text{ for all $z\in \bbC\backslash [0,\infty)$}. 
\lb{COMPK}
\end{split} 
\end{align} 
\end{theorem}

While \eqref{ESSK} (resp., \eqref{COMPK}) is a classical result of Krein \cite{Kr47}, the more general 
fact \eqref{Barr-5} is due to Alonso and Simon \cite{AS80}. 

Subsequently,  in Theorem \ref{t4.1}, we will slightly improve on \eqref{COMPK} to the effect that 
$z\in \bbC\backslash [0,\infty)$ can be replaced by $z\in \bbC\backslash [\varepsilon,\infty)$. More 
importantly, Theorem \ref{t4.1} recalls the extension of the result \eqref{COMPK}, to trace ideals 
$\cB_p(\cH)$, $p \in (0, \infty]$ (recorded earlier in \cite{AGMST10} for $p \in [1,\infty]$). 

Concluding this section, we point out that a great variety of additional results 
for Krein--von Neumann (and Krein--von Neumann-type) extensions can be found, for instance, in 
\cite[\S 109]{AG81a}, \cite{AS80}, \cite{AN70}, \cite{Ar98}, \cite{Ar00}, \cite[Chs.~9, 10]{ABT11}, \cite{AHSD01}, \cite{AK11}, \cite{AT02}, \cite{AT03}, \cite{AT05}, \cite{AT09},  \cite{AGLMS17}, \cite{AGMST10}, \cite{AGMST13}, \cite{AGMT10}, \cite{BGM22}, \cite[Ch.~5]{BHS20}, \cite{BC05}, \cite{DM91}, \cite{DM95}, \cite[Ch.~8]{DM17}, \cite{EM05}, \cite{EMP04}, \cite{EMMP07}, \cite[Part III]{Fa75}, \cite{FN09}, \cite{FN10}, \cite{FGKLNS22}, \cite[Sect.~3.3]{FOT11}, \cite{GMO21}, \cite{GM11}, \cite[App.~D]{GNZ24}, \cite{Gr83}, \cite{Gr06}, \cite[Ch.~13]{Gr09}, \cite{Gr12}, \cite{Ha57}, \cite{HK09}, \cite{HMD04}, \cite{HSDW07}, \cite{HSS09}, \cite{Ko88}, \cite[Ch.~3]{Ko99}, \cite{KO77}, \cite{KO78}, \cite{Ne83}, \cite{PS96}, \cite{SS03}, \cite{SS03}, \cite[Sects.~13.3, 14.8]{Sc12}, \cite{Si98}, \cite{Sk79}, \cite{St96}, \cite{Ts80}, \cite{Ts81}, \cite{Ts92}, \cite{Vo15}, and the references therein. In particular, we refer to the encyclopedic treatment of self-adjoint extensions and boundary triplets treated in the monograph \cite[Chs.~2--5]{BHS20}. 

For the unitary equivalence between the reduced Krein--von Neumann operator $\hat S_K$ and an abstract buckling problem operator, see \cite{AGMST10}, \cite{AGMST13}, \cite{AGMT10}, \cite[App.~D.4]{GNZ24}.

Since the literature on the Friedrichs extension is so voluminous, we just refer to Friedrichs \cite{Fr34}, \cite{Fr35}, and Freudenthal \cite{Fr36}, and the following discussions in the monographs, \cite[Ch.~5]{BHS20}, \cite[Sect.~14.8]{BSU96}, \cite[p.~164--167, 177]{EE18}, \cite[Part III]{Fa75}, \cite[Ch.~3]{GG91}, \cite[Sect.~4.4, Chs.~12, 13]{Gr09},
\cite[Sect.~VI.2, ]{Ka80}, \cite[Sect.~X.3]{RS75}, \cite[Chs.~10, 11, 13, 14]{Sc12}, \cite[Sect.~5.5]{We80}, \cite[Sect.~4.2]{We00}.

\section{On Krein-Type Resolvent Formulas \\ and Donoghue-Type $M$-Operators} \lb{s3}

In this section we discuss various Krein-type resolvent formulas and start by recalling some pertinent facts derived in \cite{GKMT01} and \cite{GMT98}. 

Temporarily, we make the following assumption:

\begin{hypothesis} \lb{h3.1}
Suppose that $S$ is a densely defined, symmetric, closed operator in $\cH$ with deficiency indices $(r,r)$, $r \in \bbN \cup \{\infty\}$. 
\end{hypothesis}

Let $\wti S$ be a self-adjoint extension (not necessarily bounded from below) of $S$ and define 
\begin{equation}\lb{e1}
U_{{\wti S},z,z_0}=I_{\cH}+(z-z_0)\big(\wti S-z I_{\cH}\big)^{-1}
= \big(\wti S-z_0 I_{\cH}\big)\big(\wti S-z I_{\cH}\big)^{-1}, 
\quad z,z_0\in\rho \big(\wti S\big).
\end{equation}
One then verifies
\begin{equation}\lb{e2}
U_{{\wti S},z_0,z_0} = I_{\cH}, \quad U_{{\wti S},z_0,z_1}U_{{\wti S},z_1,z_2}
=U_{{\wti S},z_0,z_2},   \quad  z_0, z_1, z_2\in \rho \big(\wti S\big), 
\end{equation}
and
\begin{equation}\lb{e3}
U_{{\wti S},z,z_0} \ker (S^*-z_0 I_{\cH})=\ker (S^*-z I_{\cH}), \quad z,z_0\in\rho \big(\wti S\big). 
\end{equation}

The deficiency subspaces
$\cN_{\pm}$ of $S$ are given by
\begin{equation}
\cN_{\pm}=\ker({S}^*\mp i I_{\cH}), \quad \dim (\cN_{\pm})= {\rm def} (S) = r \in \bbN \cup \{\infty\}, 
\lb{2.59}
\end{equation}
and for any self-adjoint extension $\wti S$ of $S$ in $\cH$ ,
the corresponding unitary Cayley transform $C_{\wti S}$ in $\cH$ is defined by
\begin{equation}
C_{\wti S} = U_{{\wti S},i,-i} = \big(\wti S+i I_{\cH}\big)\big(\wti S-i I_{\cH}\big)^{-1},
\lb{2.60}
\end{equation}
implying
\begin{equation}
C_{\wti S}\cN_-=\cN_+.
\lb{2.61}
\end{equation}

Let $S_\ell$, $\ell=1,2$,  be two distinct, relatively prime self-adjoint extensions 
of $S$, that is, $\dom (S)=\dom (S_1)\cap \dom (S_2)$ (again, at this instance, 
$S_1$ and $S_2$ need not be bounded from below). Associated with $S_1$ and $S_2$ we 
introduce $P_{1,2}(z)\in\cB(\cH)$ by
\begin{align}
\begin{split}
P_{1,2}(z)&= (S_1-z I_{\cH})(S_1-i I_{\cH})^{-1} \big[(S_2-z I_{\cH})^{-1} 
-(S_1-z I_{\cH})^{-1}\big]    \\
& \quad \times (S_1-z I_{\cH})(S_1+i I_{\cH})^{-1}, 
\quad z\in \rho(S_1)\cap\rho(S_2).     \lb{2.62}
\end{split}
\end{align}
We refer to Lemma 2 of \cite{GMT98} and \cite{Sa65} for a detailed discussion of
$P_{1,2}(z)$. Here we only mention the following properties of $P_{1,2}(z)$,  
$z\in \rho(S_1)\cap\rho(S_2)$,
\begin{align}
&P_{1,2}(z)|_{\cN_+^\bot} =0, \quad P_{1,2}(z)\cN_+
\subseteq \cN_+, \lb{2.63} \\
&\ol{\ran (P_{1,2}(i))}= \cN_+, \quad \ran(P_{1,2}(z)|_{\cN_+}) 
\, \text{ is independent of } \, z\in \rho(S_1)\cap\rho(S_2),
\lb{2.64}\\
&P_{1,2}(i)|_{\cN_+} =(i/2) \big(I_{\cH} - C_{S_2}(C_{S_1})^{-1}\big)\big|_{\cN_+}
=(i/2) \big(I_{\cN_+}+e^{-2i\alpha_{1,2}}\big)
\lb{2.65}
\end{align}
for some self-adjoint (possibly unbounded) operator $\alpha_{1,2} $ in $\cN_+$.

Given a self-adjoint extension $\wti S$ of $S$ and a closed linear subspace 
$\cN$ of $\cN_+$, $\cN\subseteq \cN_+$, the Donoghue-type $M$-operator 
$M_{\wti S,\cN}^{D}(z)
\in\cB(\cN)$ associated with the pair $\big(\wti S,\cN\big)$  is defined by
\begin{align}
M_{\wti S,\cN}^{D}(z)&=P_\cN \big(z\wti S+I_\cH\big)\big(\wti S-z I_{\cH}\big)^{-1} P_\cN|_\cN  \no \\
&=zI_\cN+(1+z^2)P_\cN \big(\wti S-z I_{\cH}\big)^{-1} P_\cN|_\cN\,, \quad  
z\in \bbC\backslash \bbR,
\lb{2.66}
\end{align}
with $I_\cN$ the identity operator in $\cN$ and $P_\cN$ the orthogonal projection in 
$\cH$ onto $\cN$.

\begin{remark} \lb{r3.1}
At this point we pause for a moment and note that the Donoghue-type $M$-operator $M_{\wti S,\cN}^{D}(\dott)$ in \eqref{2.66} can be viewed as the Weyl--Titchmarsh operator function in the theory of boundary triplets (see \cite[Chs.~2--5]{BHS20} for an exhaustive treatment of the latter) as detailed in Appendix~A of \cite{BGHN16}, see, in particular, \cite[Proposition~A.4]{BGHN16}. \hfill $\diamond$
\end{remark}

Next, one verifies (cf. Lemma 4 in \cite{GMT98}), 
\begin{align}
\big(P_{1,2}(z)|_{\cN_+}\big)^{-1} &=\big(P_{1,2}(i)|_{\cN_+}\big)^{-1}
-(z-i) P_{\cN_+}(S_1+i I_{\cH})(S_1-z I_{\cH})^{-1} P_{\cN_+}  \no \\
&=\tan (\alpha_{1,2})-M_{S_1, \cN_+}^{D}(z),
\quad z\in \rho(S_1),     \lb{2.67}
\end{align} 
where
\begin{equation}
C_{S_2}(C_{S_1})^{-1}\big|_{\cN_+} = -e^{-2i\alpha_{1,2}}.
\lb{2.68}
\end{equation}
Following Saakjan \cite{Sa65} (in a version presented in Theorem 5 and 
Corollary 6 in \cite{GMT98}), and using the notions introduced in 
\eqref{2.62}--\eqref{2.68}, a special case of Krein's resolvent formula then reads as 
follows (cf.\ \cite{GMT98}, \cite{Sa65}):

\begin{theorem} \lb{t3.2} 
Assume Hypothesis \ref{h3.1} and let $S_1$ and $S_2$ be relatively prime self-adjoint
extensions of $S$. Then  
\begin{align}
(S_2-z I_{\cH})^{-1}&=(S_1-z I_{\cH})^{-1}   \no \\
& \quad +(S_1-i I_{\cH})(S_1-z I_{\cH})^{-1}P_{1,2}(z)
(S_1+i I_{\cH})(S_1-z I_{\cH})^{-1} \lb{2.69} \\
&=(S_1-z I_{\cH})^{-1}+(S_1-i I_{\cH})(S_1-z I_{\cH})^{-1} P_{\cN_+}
\lb{2.70}   \\
&\quad \times \big[\tan (\alpha_{1,2})- M_{S_1,\cN_+}^{D}(z)\big]^{-1}
P_{\cN_+}(S_1+i I_{\cH})(S_1-z I_{\cH})^{-1},    \no \\
& \hspace*{6.65cm} z\in \rho(S_1)\cap\rho(S_2).   \no 
\end{align}
\end{theorem}

Next, we recall that $M_{\wti S,\cN}^{D}$ and hence $P_{1,2}(z)|_{\cN_+}$
 and $-(P_{1,2}(z)|_{\cN_+})^{-1}$ (cf.\ \eqref{2.67}), are operator-valued Nevanlinna--Herglotz functions.

\begin{theorem} [\cite{GNWZ19}, Theorem~5.3] \lb{t3.3} 
Assume Hypothesis \ref{h3.1} and 
let $\wti S$ be a self-adjoint extension of  $S$ with orthogonal family of spectral
projections $\{E_{\wti S}(\lambda)\}_{\lambda\in \bbR}$, $\cN$ a closed linear subspace of 
$\cN_+$. Then the Donoghue-type $M$-operator $M_{\wti S,\cN}^{D}(z)$ is analytic for 
$z\in \bbC\backslash\bbR$ and 
\begin{align}
& [\Im(z)]^{-1} \Im\big(M_{\wti S,\cN}^{D}(z)\big) \geq
2 \Big[\big(|z|^2 + 1\big) + \big[\big(|z|^2 -1\big)^2 + 4 (\Re(z))^2\big]^{1/2}\Big]^{-1} I_{\cN},    \no \\
& \hspace*{9.5cm}  z\in \bbC\backslash \bbR,    \lb{2.71}
\end{align}
see \cite{GNWZ19} $($cf.\ also \cite{GMT98}$)$.
In particular,
\begin{equation}
\big[\Im\big(M_{\wti S,\cN}^{D}(z)\big)\big]^{-1} \in \cB(\cN), \quad
z\in \bbC\backslash \bbR,     \lb{2.71A}
\end{equation}
and $M_{\wti S,\cN}^{D}(z)$ is a $\cB(\cN)$-valued Nevanlinna--Herglotz function that 
admits the following representation valid in the strong operator topology of $\cN$,
\begin{equation}
M_{\wti S,\cN}^{D}(z)= \int_\bbR d\Omega^D_{\wti S,\cN}(\lambda)\, \bigg[\f{1}{\lambda-z} -
\f{\lambda}{1+\lambda^2}\bigg], \quad z\in\bbC\backslash\bbR,
\lb{2.72}
\end{equation}
where
\begin{align}
&\Omega^D_{\wti S,\cN}(\lambda)=(1+\lambda^2)
\big(P_\cN E_{\wti S}(\lambda)P_\cN|_\cN\big), \quad \lambda \in\bbR, 
\lb{2.73} \\
&\int_\bbR d\Omega^D_{\wti S,\cN}(\lambda)\, (1+\lambda^2)^{-1}=I_\cN,
\lb{2.74} \\
&\int_\bbR d(f,\Omega^D_{\wti S,\cN} (\lambda)f)_\cN=\infty \, \text{ for all } \, 
f\in \cN\backslash\{0\}.
\lb{2.75}
\end{align}
\end{theorem}

We also recall (a special case of) the principal result of \cite{GMT98}, the linear fractional transformation relating the Donoghue-type  $M$-operators associated with different self-adjoint extensions of $S$:

\begin{theorem} \lb{t3.4}
Assume Hypothesis \ref{h3.1} and 
let $S_1$ and $S_2$ be relatively prime self-adjoint extensions of $S$. Then, for 
$z\in\rho(S_1)\cap\rho(S_2)$, 
\begin{align}
M_{S_2, \cN_+}^{D}(z)&=\big[P_{1,2}(i)|_{\cN_+} +\big(I_{\cN_+}+i P_{1,2}(i)|_{\cN_+}\big] M_{S_1, \cN_+}^{D}(z)\big)  \no \\
& \quad \times \big[\big(I_{\cN_+}+ iP_{1,2}(i)|_{\cN_+}\big)-P_{1,2}(i)|_{\cN_+}
M_{S_1, \cN_+}^{D}(z)\big]^{-1},
\lb{2.76}
\end{align}
where
\begin{align}
P_{1,2}(i)|_{\cN_+}&=(i/2) \big(I_\cH -C_{S_2}(C_{S_1})^{-1}\big)\big|_{\cN_+},
\lb{2.77} \\
I_{\cN_+}+iP_{1,2}(i)|_{\cN_+}
&=(1/2) \big(I_\cH +C_{S_2}(C_{S_1})^{-1}\big)\big|_{\cN_+}.
\lb{2.78}
\end{align}
Introducing the self-adjoint operator $\alpha_{1,2}$ in $\cN_+$ by 
\begin{equation}\lb{2.79}
e^{-2i\alpha_{1,2}}=-C_{S_2}(C_{S_1})^{-1}\big|_{\cN_+},
\end{equation}
\eqref{2.76} can be rewritten as
\begin{align}
\begin{split} 
M_{S_2, \cN_+}^{D}(z) &=e^{-i\alpha_{1,2}} \big[\cos (\alpha_{1,2})
+\sin (\alpha_{1,2}) M_{S_1, \cN_+}^{D}(z)\big]    \\
& \quad \times \big[\sin (\alpha_{1,2}) - \cos (\alpha_{1,2})
M_{S_1,\cN_+}^{D}(z)\big]^{-1}e^{i\alpha_{1,2}}. \lb{2.80}
\end{split} 
\end{align}
\end{theorem}

Abstract Weyl--Titchmarsh operators of the type $M_{\wti S,\cN}^{D}(z)$ have attracted considerable attention in the literature. In the case $r=1$ they have been introduced by Donoghue \cite{Do65}. The interested reader can find a variety of additional results, for instance, in \cite{BMN17}, \cite{DM87}--\cite{DMT88}, \cite{KO77}, \cite{KO78}, \cite{Ma92a}, \cite{Ma92b}, and the references cited therein.

Next we recall a slight refinement of a result of Krein \cite{Kr47} (see also \cite{AT82}, \cite{Ts80}, \cite{Ts81}). We will use an efficient summary of Krein's result due to Skau \cite{Sk79} (cf.~also \cite{LT77}), in a version that appeared in \cite{GKMT01}:

\begin{theorem}\lb{t3.5}
Assume Hypothesis \ref{h2.1} and introduce the deficiency subspaces $\cN_\pm=\ker({S}^*\mp i I_{\cH})$. Suppose that $\wti S\ge 0$ is a self-adjoint extension of $S$ in $\cH $ with corresponding family of orthogonal spectral projections $\{E_{\wti S}(\lambda)\}_{\lambda\ge 0}$ and define
\begin{equation}
\Omega^D_{\wti S,\cN_+} (\lambda) = \big(1+\lambda^2\big)
\big(P_{\cN_+}E_{\wti S}(\lambda)P_{\cN_+}|_{\cN_+}\big).
\lb{2.81}
\end{equation}
Denote by $S_F$ and $S_K$ the Friedrichs
and Krein--von Neumann extension of $S$, respectively. Then \\[1mm] 
$(i)$ $\wti S=S_F$ if and only if $\int_R^\infty d \| E_{\wti S}(\lambda)u_+
\|^2_\cH \lambda=\infty$, or equivalently, if and only if 
$\int_R^\infty d(u_+, \Omega^D_{\wti S,\cN_+} (\lambda) u_+)_{\cN_+}\lambda^{-1}=\infty$
 for all $R>0$ and all $u_+\in \cN_+\backslash\{0\}$.   \\[1mm] 
$(ii)$ $\wti S=S_K$ if and only if $\int_0^R 
d \| E_{\wti S}(\lambda)u_+\vert\vert^2_\cH \lambda^{-1} =\infty$, or equivalently, 
if and only if  $\int_0^R d(u_+, \Omega^D_{\wti S,\cN_+}
(\lambda) u_+)_{\cN_+}\lambda^{-1}=\infty$ for all $R>0$ and all 
$u_+\in \cN_+\backslash\{0\}$.  \\[1mm] 
$(iii)$ $\wti S=S_F=S_K $ if and only if $\int_R^\infty d\| E_{\wti S}(\lambda)u_+
\|^2_\cH\lambda= \int_0^R d\| E_{\wti S}(\lambda)u_+ \|^2_\cH\lambda^{-1}=\infty$, 
 or equivalently, if and only if for all $R>0$ and all $u_+ \in \cN_+\backslash\{0\}$, \\
$\int_R^\infty d(u_+, \Omega^D_{\wti S,\cN_+} (\lambda) u_+)_{\cN_+}\lambda^{-1}=
\int_R^\infty d(u_+, \Omega^D_{\wti S,\cN_+} (\lambda) u_+)_{\cN_+}
\lambda^{-1}=\infty$. 
\end{theorem}

The following is a direct consequence of Theorem \ref{t3.5} (see also, \cite{DM91},
\cite{DM95}, \cite{DMT88}, \cite{KO78}, and \cite{Ts92}): 

\begin{corollary} \lb{c3.6} 
Assume Hypothesis \ref{h2.1} and introduce the deficiency subspaces 
$\cN_\pm=\ker({S}^*\mp i I_{\cH})$. Suppose that $\wti S\ge 0$ is a 
self-adjoint extension of $S$ in $\cH$. \\[1mm] 
$(i)$ $\wti S=S_F$ if and only if
$\lim_{\lambda\downarrow -\infty} \big(u_+, M_{\wti S,\cN_+}^{D}
(\lambda)u_+\big)_{\cN_+}=-\infty$ for all $u_+\in \cN_+\backslash\{0\}$. \\[1mm] 
$(ii)$ $\wti S=S_K$ if and only if
$\lim_{\lambda\uparrow 0} \big(u_+, M_{\wti S,\cN_+}^{D} (\lambda)u_+\big)_{\cN_+}= \infty$
for all $u_+\in \cN_+\backslash\{0\}$. \\[1mm] 
$(iii)$ $\wti S=S_F=S_K$ if and only if
$\lim_{\lambda\downarrow -\infty} \big(u_+, M_{\wti S,\cN_+}^{D}(\lambda)u_+\big)_{\cN_+}=-\infty$
and \\ $\lim_{\lambda\uparrow 0} \big(u_+, M_{\wti S,\cN_+}^{D} (\lambda)u_+\big)_{\cN_+}= \infty$
for all $u_+\in \cN_+\backslash \{0\}$. 
\end{corollary}

Since $S_K$ and $S_F$ are relatively prime by Lemma \ref{l2.9}, Theorem \ref{t3.2} applies 
and we obtain the following version of Krein's formula connecting the resolvents of $S_K$ and $S_F$: 

\begin{theorem}\lb{t3.7}
Assume Hypothesis \ref{h2.7} and recall the deficiency subspace 
$\cN_+ = \ker({S}^* - i I_{\cH})$. Then 
\begin{align}
(S_K-z I_{\cH})^{-1} &=(S_F-z I_{\cH})^{-1}   \no \\ 
& \quad +(S_F-i I_{\cH})(S_F-z I_{\cH})^{-1} P_{\cN_+}
\big[M_{S_F,\cN_+}^{D}(0) - M_{S_F,\cN_+}^{D}(z)\big]^{-1}  \no   \\
&\qquad \times P_{\cN_+}(S_F+i I_{\cH})(S_F-z I_{\cH})^{-1},  \quad 
z \in \rho(S_K)\cap\rho(S_F).   \lb{3.1}
\end{align}
\end{theorem}
\begin{proof}
By the result \eqref{2.70} (with $S_1=S_K$ and $S_2=S_F$), one only needs to identify the operator $\tan(\alpha_{1,2})$ in the present context. But 
$(S_F)^{-1} \in\cB(\cH)$ and 
\begin{equation} 
\dim(\ker(S_K)) = \dim(\ker(S^*)) = \dim(\cN_+) = r\in\bbN \cup \{\infty\}
\end{equation} 
(cf.\ \eqref{dim}), equation \eqref{2.66}, and Corollary \ref{c3.6}\,$(ii)$, make 
it plain that $\tan(\alpha_{1,2}) = M_{S_F,\cN_+}^{D}(0)$, implying \eqref{3.1}. 
\end{proof}

We note that Nenciu \cite{Ne83} hints at formula \eqref{3.1} (cf.\ his comment following the proof of Lemma 3 in \cite{Ne83}). In the case of a nonnegative half-line Schr\"odinger operators with deficiency indices (1, 1), the boundary condition characterizing the Krein--von Neumann extension was identified in terms of the Weyl--Titchmarsh function by Derkach, Malamud, Tsekanovskii in \cite{DMT88}, \cite{DMT89}. We also note that \cite{Ts92} yields the analog of \eqref{3.1} in this case. Furthermore, a general description of the Krein--von Neumann extension in terms of boundary triplets is provided in \cite{DMT89}. In this context, see also \cite[Ch.~6]{BHS20} and \cite[Ch.~6, Sect.~D.7]{GNZ24}.

Interchanging the role of $S_F$ and $S_K$ one obtains the following result:

\begin{theorem} \lb{t3.8} 
 Assume Hypothesis \ref{h2.7} and recall the deficiency subspace 
$\cN_+=\ker({S}^* - i I_{\cH})$. Then 
\begin{align}
& (S_F-z I_{\cH})^{-1} = (S_K-z I_{\cH})^{-1}   \no \\ 
& \quad +(S_K-i I_{\cH})(S_K-z I_{\cH})^{-1} P_{\cN_+}
\big[-M_{S_F,\cN_+}^{D}(0) - M_{S_K,\cN_+}^{D}(z)\big]^{-1}    \lb{3.16} \\
&\qquad \times P_{\cN_+}(S_K+i I_{\cH})(S_K-z I_{\cH})^{-1},  \quad 
z \in \rho(S_K)\cap\rho(S_F).   \no 
\end{align}
\end{theorem}
\begin{proof}
 Recalling the notation \eqref{2.60}, we first prove the relation 
 \begin{equation}
  C_{S_K}(C_{S_F})^{-1} \big|_{\cN _{+}} = \big(C_{S_F} (C_{S_K})^{-1}
  \big|_{\cN _{+}}\big)^{-1}.     \lb{3.17}
 \end{equation}
Indeed, one notes that the linear subspace $\cN_{+}$ is left invariant by the bounded operator 
$C_{S_F}(C_{S_K})^{-1}$ as 
\begin{equation} 
C_{S_F}(C_{S_K})^{-1} |_{\cN _{+}} = U_{{S_F},i,-i}U_{{S_K},-i,i} |_{\cN _{+}},
\end{equation} 
and hence using formula \eqref{e3}, one concludes that 
$C_{S_F}(C_{S_K})^{-1} \cN _{+} = \cN_{+}$, implying 
\eqref{3.17}. Since $C_{S_F}(C_{S_K})^{-1}\big|_{\cN _{+}}$ 
and $C_{S_K}(C_{S_F})^{-1}\big|_{\cN _{+}}$ are unitary, there exist two self-adjoint
operators $\alpha_{K,F}$ and $\alpha_{F,K}$ in $\cN_+$ such that 
\begin{equation}
  C_{S_F}(C_{S_K})^{-1}\big|_{\cN _{+}}=-e^{-2i\alpha_{K,F}}, \quad 
  C_{S_K}(C_{S_F})^{-1}\big|_{\cN _{+}}=-e^{-2i\alpha_{F,K}}.    \lb{3.19} 
\end{equation}
Moreover, $\tan(\alpha_{F,K}) = M_{S_F,\cN_+}^{D}(0)$ together with \eqref{3.17} imply  
\begin{equation}
\tan(\alpha_{K,F}) = - M_{S_F,\cN_+}^{D}(0) = - \tan(\alpha_{F,K}).    \lb{3.20}
\end{equation}
Formula \eqref{3.16} thus follows from \eqref{2.70} and \eqref{3.20}. 
\end{proof}

We also mention the following fact:

\begin{lemma} \lb{l3.9} 
 Assume Hypothesis \ref{h2.7}. Then for every $u_+ \in\ker(S^*-iI_{\cH}) = \cN_+$ one has 
 \begin{equation} 
 \slim_{z\rightarrow 0}zU_{{S_K},z,i} u_+ =iP_{\ker(S^*)} u_+.   \lb{3.21} 
 \end{equation} 
\end{lemma}
\begin{proof}
Let $u_+ \in \ker(S^*-iI_{\cH})$ then $U_{{S_K},z,i} u_+ = u_+ +(z-i)(S_K - z I_{\cH})^{-1} u_+$. The Laurent 
expansion of the resolvent $(S_K - z I_{\cH})^{-1}$ around $z = 0$ is of the form 
\begin{equation} 
(S_K - z I_{\cH})^{-1} = - \frac{P_{\ker(S^*)}}{z}+\sum\limits_{n=0}^{\infty}  z^n 
\big({\hatt{S}_K}\big)^{-n-1}, \quad P_{\ker(S_K)} = P_{\ker(S^*)}. 
\lb{3.29} 
\end{equation} 
Using \eqref{3.29} and the definition of $U_{{S_K},z,i}$ one obtains 
\begin{equation} 
zU_{{S_K},z,i} u_+ = z u_+ - (z-i)P_{\ker(S^*)} u_+ + z(z-i)\sum\limits_{n=0}^{\infty} z^n 
\big({\hatt{S}_K}\big)^{-n-1} u_+,
\end{equation} 
and hence \eqref{3.21} as $z\rightarrow 0$.
\end{proof}

We conclude this section with a result inspired by \cite[Theorem~3.1]{GG91}:

\begin{theorem} \lb{t3.10}
Assume Hypothesis \ref{h3.1} and let $S_j$, $j =0,1,2$, be self-adjoint extensions of $S$ such that $S_1$ and $S_0$ as well as $S_2$ and $S_0$ are relatively prime. In addition, let $p \in (0,\infty]$. Then 
\begin{align}
&\big[(S_2 - z I_{\cH})^{-1} - (S_1 - z I_{\cH})^{-1}\big] \in \cB_p(\cH) \, \text{ for some}     \no \\ 
& \qquad \text{$($and hence for all\,$)$ $z \in \rho(S_1)\cap \rho(S_2)$}      \lb{3.38} \\
& \quad \text{if and only if } \, \big[e^{2i\alpha_{0,2}} - e^{2i\alpha_{0,1}}\big] \in \cB_p(\cN_+).    \no
\end{align} 
\end{theorem}
\begin{proof}
Applying Theorem \ref{t3.2} for $z=i$ to the pairs $(S_j,S_0)$, $j=1,2$, and subtracting the results yields upon noting that $M^D_{S_j,\cN_+}(i) = i I_{\cN_+}$, 
\begin{align}
& (S_2 - i I_{\cH})^{-1} - (S_1 - i I_{\cH})^{-1} = P_{\cN_+} \Big\{\big[\tan(\alpha_{0,2}) - i I_{\cN_+}\big]^{-1}  \no \\
& \qquad - \big[\tan(\alpha_{0,1}) - i I_{\cN_+}\big]^{-1} \Big\} P_{\cN_+} (S_0 + i I_{\cN_+}) S_0 - i I_{\cN_+})^{-1}    \no \\
& \quad = P_{\cN_+} \Big\{\big[\tan(\alpha_{0,2}) - i I_{\cN_+}\big]^{-1} - \big[\tan(\alpha_{0,1}) - i I_{\cN_+}\big]^{-1}\Big\}P_{\cN_+} C_{S_0}   \no \\
& \quad =  2 i P_{\cN_+} \big[e^{2i\alpha_{0,2}} + I_{\cN_+}\big]^{-1} \big[e^{2i\alpha_{0,2}} - e^{2i\alpha_{0,1}}\big] 
\big[e^{2i\alpha_{0,1}} + I_{\cN_+}\big]^{-1} P_{\cN_+} C_{S_0}.     \lb{3.39} 
\end{align} 
Since the Cayley transform $C_{S_0}$ is unitary and $S_j \neq S_0$ as $(S_j,S_0)$, $j=1,2$, are relatively prime, one concludes that 
\begin{equation} 
\alpha_{0,j} \neq (\pi/2) I_{\cN_+}, \quad j=1,2, 
\end{equation} 
and  
\begin{equation} 
\big[\tan(\alpha_{0,j}) - i I_{\cN_+}\big]^{-1}, \, 
\big[e^{2i\alpha_{0,j}} + I_{\cN_+}\big]^{-1}, \, \big[e^{2i\alpha_{0,j}} + I_{\cN_+}\big] \in \cB(\cN_+), \quad j=1,2.   
\end{equation} 
Thus, relation \eqref{3.39} completes the proof of \eqref{3.38} for $z=i$. The identity 
(cf., e.g., \cite[p.~178]{We80}) 
\begin{align}
& (S_2 - z I_{\cH})^{-1} - (S_1 - z I_{\cH})^{-1} = (S_2 - z_0 I_{\cH})(S_2 - z I_{\cH})^{-1}     \no \\
& \quad \times \big[(S_2 - z_0 I_{\cH})^{-1} - (S_1 - z_0 I_{\cH})^{-1}\big] (S_1 - z_0 I_{\cH})(S_1 - z I_{\cH})^{-1},   \\
& \hspace*{6.6cm} z, z_0 \in \rho(S_1) \cap \rho(S_2),    \no
\end{align}
implies that  
\begin{align}
\begin{split} 
& \big[(S_2 - z I_{\cH})^{-1} - (S_1 - z I_{\cH})^{-1}\big] \in \cB_p(\cH), \quad z \in \rho(S_1) \cap \rho(S_2),    \\
& \quad \text{if and only if } \big[(S_2 - i I_{\cH})^{-1} - (S_1 - i I_{\cH})^{-1}\big] \in \cB_p(\cH), 
\end{split} 
\end{align}
completing the argument. 
\end{proof}

\section{Some Comparisons between the Abstract Friedrichs and \\ Krein--von Neumann Extension}   \lb{s4}

In this section we draw certain comparisons between the resolvent of the abstract Friedrichs and that of the (reduced) Krein--von Neumann 
extension. 

As a warmup we start with the following extension of \eqref{COMPK}, part $(i)$ of which essentially appeared in \cite{AGMST10}:

\begin{theorem}\lb{t4.1}
Assume Hypothesis \ref{h2.7} and let $p\in (0,\infty]$. \\[1mm]
$(i)$ Then 
\begin{align}
& (S_F - z_0 I_{\cH})^{-1} \in \cB_p(\cH) 
\, \text{ for some $z_0\in \bbC\backslash [\varepsilon,\infty)$,}   \no \\
& \quad \text{implies}         \lb{4.1} \\
& (S_K - zI_{\cH})^{-1}[I_{\cH} - P_{\ker(S_K)}] \in \cB_p(\cH) 
 \, \text{ for all $z\in \bbC\backslash [\varepsilon,\infty)$}     \no \\
& \text{$\big($equivalently, $\big(\hatt S_K - zI_{\hatt\cH}\big)^{-1} \in \cB_p\big(\hatt\cH\big)\big)$ for all $z\in \bbC\backslash [\varepsilon,\infty)$$\big)$.}   \no
\end{align}
In fact, the $\ell^p(\bbN)$-based trace ideal $\cB_p(\cH)$ of $\cB(\cH)$ can be replaced by any 
two-sided symmetrically normed ideal of $\cB(\cH)$. \\[1mm]
$(ii)$ If $S$ has a self-adjoint extension $\wti S$ with $0 \in \rho\big(\wti S\big)$, then 
\begin{equation}
 \big(\wti S\big)^{-1} \in \cB_p(\cH) \, \text{ implies } \,\big(\hatt S_K\big)^{-1} \in \cB_p\big(\hatt \cH\big).   \lb{4.2}
\end{equation}
\end{theorem}
\begin{proof} 
For the $\ell^p$-based trace ideals $\cB_p(\cH)$ of $\cB(\cH)$ the assertion \eqref{4.1} is an immediate consequence of inequality \eqref{Barr-5}. \\[1mm]
Assertion \eqref{4.1} also immediately follows from \eqref{2.FK}. \\[1mm]
Similarly, \eqref{4.2} is clear from \eqref{3.28a}. \\[1mm] 
Next, we provide an elementary alternative proof that applies to all two-sided symmetrically normed ideals of $\cB(\cH)$.
  
We suppose that $z \in \rho(S_K)\cap\rho(S_F)$ in this proof until \eqref{4.8}.  
Multiplying \eqref{3.1} by $[I_{\cH} - P_{\ker(S_K)}]^2=[I_{\cH} - P_{\ker(S_K)}]$, 
and using that $[I_{\cH} - P_{\ker(S_K)}]$ and $(S_K-z I_{\cH})^{-1}$ commute, one obtains 
\begin{align}
& (S_K-z I_{\cH})^{-1}[I_{\cH} - P_{\ker(S_K)}] 
= [I_{\cH} - P_{\ker(S_K)}](S_F-z I_{\cH})^{-1}[I_{\cH} - P_{\ker(S_K)}] \no \\
& \qquad + [I_{\cH} - P_{\ker(S_K)}](S_F-i I_{\cH})(S_F-z I_{\cH})^{-1} P_{\cN_+}
\big[M_{S_F,\cN_+}^{D}(0) - M_{S_F,\cN_+}^{D}(z)\big]^{-1}  \no   \\
&\qquad \quad 
\times P_{\cN_+}(S_F+i I_{\cH})(S_F-z I_{\cH})^{-1}[I_{\cH} - P_{\ker(S_K)}] \no \\
& \quad = [I_{\cH} - P_{\ker(S^*)}](S_F-z I_{\cH})^{-1}[I_{\cH} - P_{\ker(S^*)}] \no \\
& \qquad + [I_{\cH} - P_{\ker(S^*)}](S_F-i I_{\cH})(S_F-z I_{\cH})^{-1} P_{\cN_+}
\big[M_{S_F,\cN_+}^{D}(0) - M_{S_F,\cN_+}^{D}(z)\big]^{-1}  \no   \\
&\qquad \quad 
\times P_{\cN_+}(S_F+i I_{\cH})(S_F-z I_{\cH})^{-1}[I_{\cH} - P_{\ker(S^*)}]  \no \\
& \quad = [I_{\cH} - P_{\ker(S^*)}](S_F-z I_{\cH})^{-1}[I_{\cH} - P_{\ker(S^*)}]  \no \\
& \qquad + [I_{\cH} - P_{\ker(S^*)}]U_{S_F,z,i} P_{\cN_+}
\big[M_{S_F,\cN_+}^{D}(0) - M_{S_F,\cN_+}^{D}(z)\big]^{-1}   \lb{4.3} \\
& \qquad \quad \times P_{\cN_+}U_{S_F,z,-i}[I_{\cH} - P_{\ker(S^*)}],    \no  
\end{align}
since $\ker(S_K)=\ker(S^*)$ by \eqref{Fr-4Tf}. Here we used the notation (cf.\ \eqref{e1})
\begin{equation}
U_{S_F,z,i} = (S_F-i I_{\cH})(S_F-z I_{\cH})^{-1}, \quad 
U_{S_F,z,-i} = (S_F+i I_{\cH})(S_F-z I_{\cH})^{-1},     \lb{4.4}
\end{equation}
and $\cN_+ = \ker(S^* - i I_{\cH})$ in \eqref{4.3}. 

Next, one computes for $0 \neq |z|$ sufficiently small,  
\begin{align}
&  [I_{\cH} - P_{\ker(S^*)}]U_{S_F,z,i} P_{\cN_+} 
 = [I_{\cH} - P_{\ker(S^*)}] \big[I_{\cH} + (z-i)(S_F-z I_{\cH})^{-1}\big] P_{\cN_+}  
 \no \\
& \quad = [I_{\cH} - P_{\ker(S^*)}] 
\Big[I_{\cH} + (z-i)\big[I_{\cH} -z(S_F)^{-1}\big]^{-1}(S_F)^{-1}\Big] 
P_{\cN_+}  \no \\
& \quad = [I_{\cH} - P_{\ker(S^*)}] \bigg[\big[I_{\cH} - i (S_F)^{-1}\big] + i (S_F)^{-1}    \no \\
& \hspace*{3.2cm} + (z-i) \bigg(\sum_{n=0}^\infty z^n (S_F)^{-n}\bigg) (S_F)^{-1}\bigg] P_{\cN_+}  \no \\
& \quad = z [I_{\cH} - P_{\ker(S^*)}] \big[I_{\cH} - i (S_F)^{-1}\big]  
\bigg(\sum_{n=0}^\infty z^n (S_F)^{-n}\bigg) (S_F)^{-1} P_{\cN_+},   \lb{4.5} 
\end{align}
where we used 
\begin{equation}
U_{S_F,0,i} P_{\cN_+} = \big(I_{\cH} - i (S_F)^{-1}\big) P_{\cN_+}   
= P_{\ker(S^*)} \big(I_{\cH} - i (S_F)^{-1}\big) P_{\cN_+},       \lb{4.6} 
\end{equation}
as a consequence of $U_{S_F,0,i} \ker(S^* - i I_{\cH}) = \ker(S^*)$ (cf.\ \eqref{e3}). 
Similarly, one gets  
\begin{align}
& P_{\cN_+}U_{S_F,z,-i}[I_{\cH} - P_{\ker(S^*)}] 
= \big[[I_{\cH} - P_{\ker(S^*)}](U_{S_F,z,-i})^* P_{\cN_+}\big]^*  \no \\
& \quad = \big[[I_{\cH} - P_{\ker(S^*)}] U_{S_F,{\ol z},i}) P_{\cN_+}\big]^*  \no \\
& \quad = z P_{\cN_+} (S_F)^{-1} \bigg(\sum_{n=0}^\infty z^n (S_F)^{-n}\bigg)
\big[I_{\cH} + i (S_F)^{-1}\big] [I_{\cH} - P_{\ker(S^*)}].     \lb{4.7}
\end{align}

Combining \eqref{4.3}, \eqref{4.5}, and \eqref{4.7} one obtains for $0 \neq |z|$ sufficiently small,  
\begin{align}
& (S_K-z I_{\cH})^{-1}[I_{\cH} - P_{\ker(S_K)}]    
= [I_{\cH} - P_{\ker(S^*)}](S_F-z I_{\cH})^{-1}[I_{\cH} - P_{\ker(S^*)}]  \no \\
& \quad + [I_{\cH} - P_{\ker(S^*)}] \big[I_{\cH} - i (S_F)^{-1}\big]  
\bigg(\sum_{n=0}^\infty z^n (S_F)^{-n}\bigg) (S_F)^{-1} P_{\cN_+}   \no \\
& \qquad \times z^2 \big[M_{S_F,\cN_+}^{D}(0) - M_{S_F,\cN_+}^{D}(z)\big]^{-1}   \lb{4.8} \\
& \qquad \times P_{\cN_+} (S_F)^{-1} \bigg(\sum_{n=0}^\infty z^n (S_F)^{-n}\bigg)
\big[I_{\cH} + i (S_F)^{-1}\big] [I_{\cH} - P_{\ker(S^*)}] \in \cB_p(\cH),      \no 
\end{align}
since by hypothesis $(S_F-z I_{\cH})^{-1}, \, (S_F)^{-1} \in \cB_p(\cH)$ for some  
$p\in (0,\infty]$. Analytic continuation using the resolvent equation for $S_K$ then yields \eqref{4.1} for all $z\in \bbC\backslash [0,\infty)$. Next, noting that \eqref{4.4}  
\big(with $\cN = \cN_+$ and $\wti S = S_F$\big) and $(S_F)^{-1}\in\cB(\cH)$ imply  
\begin{equation}
\f{d}{dz}M_{S_F,\cN_+}^{D}(z)\bigg|_{z=0} = I_{\cN_+} 
+ P_{\cN_+}(S_F)^{-2} P_{\cN_+} \big|_{\cN_+} \geq  I_{\cN_+},     \lb{4.9}
\end{equation}
one obtains 
\begin{align}
& - z \big[M_{S_F,\cN_+}^{D}(0) - M_{S_F,\cN_+}^{D}(z)\big]^{-1} 
= \bigg[\bigg(\f{d}{dz}M_{S_F,\cN_+}^{D}(z)\bigg|_{z=0}\bigg)^{-1} + \Oh(z)\bigg]
 \in \cB(\cN_+)   \no \\
& \hspace*{6.7cm}  \text{ for $0 \neq |z|$ sufficiently small.}    \lb{4.10}
\end{align}
(Here the symbols $\Oh(1)$, $\Oh(z)$, etc., represent bounded operators in $\cH$ 
or $\cN_+$ whose norms are of order $\Oh(1)$, $\Oh(z)$, etc., as $z\to 0$.)
Thus, the right-hand side of \eqref{4.8} has an analytic 
continuation to $z=0$ and hence so does the left-hand side (cf.\ also Remark \ref{r4.2} below). 
This is of course related to the concept of the reduced resolvent of $S_K$ for the eigenvalue 
$0$ (cf.\ \cite[Sect.\ III.6.5]{Ka80}).  

Finally, an additional analytic 
continuation yields \eqref{4.1} for all $z\in \bbC\backslash [\varepsilon,\infty)$. 

Clearly the same proof works with $\cB_p(\cH)$ replaced by any symmetrically normed ideal 
of $\cB(\cH)$.
\end{proof}

\begin{remark} \lb{r4.2}
It is illustrative to complete the computation of $(S_K-z I_{\cH})^{-1}[I_{\cH} - P_{\ker(S_K)}]$ for 
$0 \neq |z|$ sufficiently small. Continuing \eqref{4.8} and \eqref{4.10} one obtains 
for $0 \neq |z|$ sufficiently small, 
\begin{align}
& (S_K-z I_{\cH})^{-1}[I_{\cH} - P_{\ker(S_K)}] 
= [I_{\cH} - P_{\ker(S_K)}](S_F-z I_{\cH})^{-1}[I_{\cH} - P_{\ker(S_K)}] \no \\
& \qquad - z [I_{\cH} - P_{\ker(S^*)}] \big[I_{\cH} - i (S_F)^{-1}\big]  
\bigg(\sum_{n=0}^\infty z^n (S_F)^{-n}\bigg) (S_F)^{-1} P_{\cN_+}  \no \\
& \qquad \quad \times 
\bigg(\f{d}{dz}M_{S_F,\cN_+}^{D}(z)\bigg|_{z=0}\bigg)^{-1} [I_{\cN_+} + \Oh(z)]  \no \\
& \qquad \quad \times  
P_{\cN_+} (S_F)^{-1} \bigg(\sum_{n=0}^\infty z^n (S_F)^{-n}\bigg)
\big[I_{\cH} + i (S_F)^{-1}\big] [I_{\cH} - P_{\ker(S^*)}]    \no \\
& \quad = [I_{\cH} - P_{\ker(S_K)}](S_F)^{-1}[I_{\cH} + \Oh(z)]
[I_{\cH} - P_{\ker(S_K)}]   \no \\
& \qquad - z [I_{\cH} - P_{\ker(S^*)}] \big[I_{\cH} - i (S_F)^{-1}\big]  
[I_{\cH} + \Oh(z)] (S_F)^{-1} P_{\cN_+}  \no \\
& \qquad \quad \times 
\bigg(\f{d}{dz}M_{S_F,\cN_+}^{D}(z)\bigg|_{z=0}\bigg)^{-1} [I_{\cN_+} + \Oh(z)]  \no \\
& \qquad \quad \times 
P_{\cN_+} (S_F)^{-1} [I_{\cH} + \Oh(z)]
\big[I_{\cH} + i (S_F)^{-1}\big] [I_{\cH} - P_{\ker(S^*)}].  \lb{4.11}
\end{align}
One notes that as $z\to 0$, \eqref{4.11} yields once more relation \eqref{2.FK}. 
\hfill $\diamond$
\end{remark} 

We reproduced the elementary but somewhat long argument in part $(iv)$ of the proof of Theorem \ref{t4.1} since it makes close contact with the material discussed in Section \ref{s3}. However, \eqref{2.FK} is a powerful identity that can directly be exploited to prove the following extension, and, in a sense, completion of Theorem \ref{t4.1}:

\begin{theorem} \lb{t4.3}
Assume Hypothesis \ref{h2.7} and let $p\in (0,\infty]$. Recalling that 
\begin{equation}
\hatt \cH = [\ker (S^*)]^{\bot} = [I_{\cH} - P_{\ker(S^*)}] \cH 
= [I_{\cH} - P_{\ker(S_K)}] \cH = [\ker (S_K)]^{\bot},    \lb{4.12}
\end{equation}
the following items $(i)$--$(iv)$ are equivalent: \\[1mm]
$(i)$ \hspace*{0.5mm} $\big(\hatt S_K\big)^{-1} \in \cB_p\big(\hatt \cH\big)$. \\[1mm] 
$(ii)$ \hspace*{0.001mm} $[I_{\cH} - P_{\ker(S_K)}] (S_F)^{-1} [I_{\cH} - P_{\ker(S_K)}] \in \cB_p(\cH)$. \\[1mm] 
$(iii)$ $[I_{\cH} - P_{\ker(S_K)}] (S_F)^{-1/2} \in \cB_{2p}(\cH)$. \\[1mm] 
$(iv)$ \hspace*{0.00001mm}  $(S_F)^{-1/2} [I_{\cH} - P_{\ker(S_K)}] \in \cB_{2p}(\cH)$. \\[1mm] 
In particular, 
\begin{equation}
S_F^{-1} \in \cB_p\big(\cH\big) \, \text{ implies } \, \big(\hatt S_K\big)^{-1} \in \cB_p\big(\hatt \cH\big).   \lb{4.13}
\end{equation}
\end{theorem}
\begin{proof}
The equivalence of items $(i)$ and $(ii)$ is clear from equality \eqref{2.FK}; the remaining equivalences with items $(iii)$ and $(iv)$ are a consequence of the well-known fact that for any $T \in \cB(\cH)$ one has for $p\in (0,\infty]$, 
\begin{align}
\begin{split} 
& T \in \cB_{2p}(\cH) \, \text{ if and only if } \, T^*T \in \cB_p(\cH),     \\
& T \in \cB_{2p}(\cH) \, \text{ if and only if } \, T^* \in \cB_{2p}(\cH).     \lb{4.14} 
\end{split} 
\end{align} 
\end{proof}

The trace ideal assertions in Theorem \ref{t4.3} extend to $(S_F)^{-1}, \big(\hatt S_K\big)^{-1}$ replaced by $(S_F - zI_{\cH})^{-1}, \big(\hatt S_K - z I_{\cH}\big)^{-1}$, $ z \in \bbC \backslash [\varepsilon, \infty)$, upon employing resolvent identities.

\begin{remark}\lb{r4.4}
The converse implication to \eqref{4.13} in Theorem \ref{t4.3} is false. This fact has recently been shown by M. Malamud \cite{Ma23} as follows: He proved that\footnote{In this context we recall that if $z \in \gamma(S)$, where $\gamma(\dott)$ denotes the field of regularity, then $\ran(S -  z I_{\cH})$ is a closed linear subspace of $\cH$ (see, e.g., \cite[Prop.~2.1\,$(iv)$]{Sc12}).} 
\begin{align}
&  P_{\ran(S+I_{\cH})} (S + I_{\cH})^{-1} \in \cB_{\infty}(\ran(S+I_{\cH}))     \no \\ 
& \quad \text{is equivalent to }     \lb{4.15} \\
& \big(\hatt S_K - z_0 I_{\hatt \cH}\big)^{-1} \in \cB_{\infty}\big(\hatt\cH\big), \quad z_0 \in \rho\big(\hatt S_K\big), \no
\end{align}
and then combined \eqref{4.15} with the fact that 
\begin{equation}
(S + I_{\cH})^{-1} \in \cB_{\infty}(\ran(S),\cH) \, \text{ does not imply } \, (S_F + I_{\cH})^{-1} \in \cB_{\infty}(\cH), 
\lb{4.16} 
\end{equation} 
a result proven earlier by S.\ Hassi, M.\ Malamud, and H.\ de Snoo \cite[Subsect.~4.8]{HMD04} (see also \cite{Ma23a}, \cite{Ma24} for concrete examples in two and three dimensions involving self-adjoint extensions of the Laplacian). These facts extend from $\cB_{\infty} (\dott)$ to $\cB_p (\dott)$, $p \in (0,\infty)$ (cf., \cite{Ma23} and \cite{Ma24}).  

In particular, recalling the decomposition \eqref{2.50} of $\cH$,  $(S_F)^{-1}$ decomposes accordingly into the $2 \times 2$ block operator matrix 
\begin{align}
& (S_F)^{-1}      \lb{4.18} \\
& \quad = \begin{pmatrix} P_{\ker(S_K)} (S_F)^{-1} P_{\ker(S_K)} & 
P_{\ker(S_K)} (S_F)^{-1} [I_{\cH} - P_{\ker(S_K)}] \\ 
[I_{\cH} - P_{\ker(S_K)}] (S_F)^{-1} P_{\ker(S_K)} & 
[I_{\cH} - P_{\ker(S_K)}] (S_F)^{-1} [I_{\cH} - P_{\ker(S_K)}] 
\end{pmatrix}.    \no
\end{align}
Thus, assuming Hypothesis \ref{h2.7} and $\big(\hatt S_K\big)^{-1} \in \cB_p\big(\hatt \cH\big)$ as in Theorem \ref{t4.3}\,$(i)$, three of the four entries in the $2 \times 2$ 
block operator matrix \eqref{4.18} lie in $\cB_p(\cH)$, the lone possible exception being the $(1,1)$-entry, 
$P_{\ker(S_K)} (S_F)^{-1} P_{\ker(S_K)}$. The hypothesis 
\begin{align} 
& \big(\hatt S_K\big)^{-1} \in \cB_p\big(\hatt \cH\big),    \no \\ 
& \quad \text{equivalently, } \\ 
& [I_{\cH} - P_{\ker(S_K)}] (S_K)^{-1} [I_{\cH} - P_{\ker(S_K)}] \in \cB_p(\cH); \quad p \in (0,\infty],    \no 
\end{align} 
in general, cannot make any assertions regarding the trace ideal properties of $P_{\ker(S_K)} (S_F)^{-1} P_{\ker(S_K)}$. Again, we refer to \cite{HMD04}, \cite{Ma92a}, \cite{Ma23}, \cite{Ma23a}, and \cite{Ma24} for additional details.   
${}$ \hfill $\diamond$ 
\end{remark} 

\medskip

\noindent
{\bf Acknowledgments.} 
We gratefully acknowledge very helpful discussions with Jussi Behrndt. F.G.\ is indebted to Tuncay Aktosun, Sergei Avdonin, Ricardo Weder, and Vjacheslav Yurko for their kind invitation to speak at the online conference AMP 2024 and to contribute to these conference proceedings. S.S.\ was supported in part by the NSF Grants  DMS-2418900 and DMS-2510344, by the Simons Foundation grant MP-TSM-00002897, by the Office of the Vice President for Research \& Economic Development (OVPRED) at Auburn University through the Research Support Program Grant, and by grant 2024154 from the United States -- Israel Binational Science Foundation (BSF), Jerusalem, Israel.


\end{document}